\documentclass[11pt,reqno]{amsart}
\usepackage{fullpage,verbatim,paralist}
\usepackage[breaklinks,pdfstartview=FitH]{hyperref}

\usepackage{amssymb,fullpage,mathrsfs,verbatim,graphicx,paralist}
\usepackage{pdfsync}

\renewcommand{\hat}{\widehat}

\renewcommand{\leq}{\leqslant}
\renewcommand{\geq}{\geqslant}

  \DeclareMathOperator{\diam}{diam}

\theoremstyle{plain}
  \newtheorem{lemma}{Lemma}
  \newtheorem{theorem}[lemma]{Theorem}

  \theoremstyle{definition}

\newcommand{\jnote}[1]{}

\newcommand{\remove}[1]{}
\newcommand{\1}{\mathbf{1}}

\begin{document}

\title{A note on mixing times of planar random walks}

\author{James R. Lee}
\thanks{J.L. partially supported by NSF CCF-0915251 and a Sloan Research Fellowship.  This work
was completed during a visit of the authors to Microsoft Research}
\address{Computer Science \& Engineering, University of Washington}
\email{jrl@cs.washington.edu}
\author{Teng Qin}
\address{Computer Science, \'Ecole normale sup\'erieure, Paris}
\email{teng.qin@ens.fr}

\begin{abstract}
We present an infinite family of finite planar graphs $\{X_n\}$ with degree at most five
and such
that for some constant $c > 0$,
$$
\lambda_1(X_n) \geq c\left(\frac{\log \diam(X_n)}{\diam(X_n)}\right)^2\,,
$$
where $\lambda_1$ denotes the smallest non-zero eigenvalue of the graph Laplacian.
This significantly simplifies a construction of Louder and Souto.

We also remark that such a lower bound cannot hold
when the diameter is replaced by the average squared distance:
There exists a constant $c > 0$
such that for any family $\{X_n\}$ of planar graphs
we have
$$
\lambda_1(X_n) \leq c \left(\frac{1}{|X_n|^2} \sum_{x,y \in X_n} d(x,y)^2\right)^{-1}\,,
$$
where $d$ denotes the path metric on $X_n$.
\end{abstract}

\maketitle

\noindent
Recently, Louder and Souto \cite{LS12} answered negatively a question of Benjamini and Curien \cite{BC11}
by showing that there is an infinite family of bounded-degree planar graphs for which the mixing time is asymptotically
less than the square of the diameter.  Their construction uses expander graphs to construct
a surface which they triangulate to arrive at a planar graph.  We present a simple family of graphs
exhibiting the same result.

\section{The Laplacian}

For a finite, undirected graph $G=(V,E)$ with edge and vertex weights $w~:~E~\to~[0,\infty)$ and $\pi : V \to (0,\infty)$, one defines the combinatorial Laplacian $L : \ell^2(V,\pi) \to \ell^2(V,\pi)$ by
$$
Lf(x) = \sum_{y : \{x,y\} \in E} \frac{w(x,y)}{\pi(x)} (f(x) - f(y))\,.
$$
Here, $\ell^2(V,\pi)$ is equipped with the inner product $\langle f,g\rangle_{\ell^2(V,\pi)} = \sum_{x \in V} \pi(x) f(x) g(x)$.
This coincides with the unweighted combinatorial Laplacian when $w \equiv 1$ and $\pi \equiv 1$.

Recall that the smallest non-zero eigenvalue of $L$ as an operator on $\ell^2(V,\pi)$ is given by
$$
\lambda_1(G) = \min_{f \in \ell^2(V,\pi)} \left\{ \frac{\sum_{\{x,y\} \in E} w(x,y) (f(x)-f(y))^2}{\sum_{x \in V} \pi(x) f(x)^2} : \sum_{x \in V} \pi(x) f(x) = 0 \right\}\,.
$$
The Cheeger constant of $G$ is defined by
$$
h(G) = \min_{S \subseteq V} \left\{ \frac{\sum_{\{x,y\} \in E} w(x,y) |\1_S(x)-\1_S(y)|}{\pi(S)} : \pi(S) \leq \pi(V)/2 \right\},
$$
where $\1_S$ denotes the characteristic function of $S$ and $\pi(S) = \sum_{v \in V} \pi(v)$ for a subset $S \subseteq V$.

We recall the discrete Cheeger inequality
\begin{equation}\label{eq:cheeger}
\lambda_1(G) \geq \frac{h(G)^2}{2\, d_{\max}},
\end{equation}
where $d_{\max} = \max_{x \in V} \left(\pi(x)^{-1} \sum_{y : \{x,y\} \in E} w(x,y)\right)$ is the maximum ``degree'' in $G$.

For a graph $G$, we use $V(G)$ and $E(G)$ to denote the vertex
and edge sets of $G$, respectively.
Unless otherwise stated, the edge and vertex weights on a graph $G$
are taken to be uniform.

\section{The construction}

Let $T_h$ denote the complete, rooted binary tree of height $h$.  Let $T_{h,k}$ denote the result
of subdividing every edge of $T_h$ by $k$.  If $T_{h,k}$ has root $r$, we write
$$
V_{\ell} = \{v \in V(T_{h,k}) : \mathrm{dist}(v,r) = \ell \}\,
$$
for the set of nodes at distance $\ell$ from $r$.
Fix an in-order traversal of $T_{h,k}$.
For each $\ell=0,1,\ldots,hk$, we add a path $P_{\ell}$ on the nodes at depth $\ell$.
Specifically, edges go between consecutive nodes of $V_{\ell}$ in the in-order traversal.
Call this final graph $\hat T_{h,k}$.  It is straightforward to verify that $\hat T_{h,k}$ is planar.

\begin{theorem}
For every $h \geq 1$ and $k=2^h$, the following bounds hold:
\begin{enumerate}
\item $\diam(\hat T_{h,k}) \geq hk$, and
\item $\lambda_1(\hat T_{h,k}) \geq \frac{1}{7 k^2} \geq \left(\frac{\log_2 \diam(\hat T_{h,k})}{6 \diam(\hat T_{h,k})}\right)^2$\,.
\end{enumerate}
\end{theorem}

\begin{proof}
The diameter bound is (i) clear, so we focus on (ii).

Let $V = V(\hat T_{h,k})$.
Consider any $f : V \to \mathbb R$
with $\sum_{x \in V} f(x) = 0$.  Define $\bar f : V \to \mathbb R$ as follows:
For $x \in V_{\ell}$,
$$\bar f(x) = \frac{1}{|V_{\ell}|} \sum_{x \in V_{\ell}} f(x)\,.$$
Observe that $\sum_{x \in V} \bar f(x)=0$ holds as well.

Now, from \eqref{eq:cheeger} applied to the graph $P_{\ell}$, which has $h(P_{\ell}) \geq 2/|V_{\ell}|$, we have
$$
\sum_{\{x,y\} \in P_{\ell}} (f(x)-f(y))^2 \geq \frac{1}{|V_{\ell}|^2} \sum_{x \in V_{\ell}} (f(x)-\bar f(x))^2\,.
$$
Using $|V_{\ell}| \leq 2^h$ and summing over all $\ell$ yields
\begin{equation}\label{eq1}
\sum_{\ell=1}^h \sum_{\{x,y\} \in E(P_{\ell})} (f(x)-f(y))^2 \geq 2^{-2h} \|f-\bar f\|^2 \geq \frac{1}{k^2} \|f-\bar f\|^2\,.
\end{equation}

That covers the ``horizontal'' edges of $\hat T_{h,k}$.  We claim that the following bound holds for the ``vertical'' edges:
\begin{equation}\label{eq2}
\sum_{\{x,y\} \in E(T_{h,k})} (\bar f(x)-\bar f(y))^2 \geq \frac{1}{6 k^2} \|\bar f\|^2\,.
\end{equation}
Since $\sum_{x \in V} \bar f(x)=0$ and $\bar f$ is constant on the level sets $\{V_{\ell}\}$, this is implied
by a lower bound on the spectral gap of the weighted quotient graph $Q_{h,k}$ of $\hat T_{h,k}$, where each level set $V_{\ell}$ is
identified to a single vertex of weight $|V_{\ell}|$.

But $Q_{h,k}$ is simply a $k$-subdivision of the weighted graph $Q_h$
which has vertex set $\{0,1,\ldots,h\}$,
edge weights $w(j,j+1)=2^{j+1}$, and vertex weights $\pi(j)=2^j$.
In particular, $\lambda_1(Q_{h,k}) = \frac{1}{k^2} \lambda_1(Q_h)$.
Hence \eqref{eq2} is implied by the lower bound
$$
\lambda_1(Q_h) = \min_{g : \{0,1,\ldots,h\} \to \mathbb R} \left\{ \frac{\sum_{j=0}^{h-1} 2^{j+1} (g(j)-g(j+1))^2}{\sum_{j=0}^h 2^j g(j)^2} :
\sum_{j=0}^{h} 2^j g(j)=0 \right\} \geq \frac{1}{6}\,.
$$
But this follows from \eqref{eq:cheeger}, because $h(Q_h) \geq 1$, which
is easily verified:  Any $S$ of weight at most half cannot contain $h$,
which has $\pi(h) > \pi(V(Q_h))/2$.
Thus for any such $S \subseteq \{0,1,\ldots,h-1\}$, if
$j = \max(S)$, then $\pi(S) \leq 2^{j} + 2^{j-1} + \cdots + 1 \leq 2^{j+1}$,
while the edge $(j,j+1)$ leaves $S$ and has $w(j,j+1)=2^{j+1}$.

Finally, we claim that
\begin{equation}\label{eq3}
\sum_{\{x,y\} \in E(T_{h,k})} (f(x)-f(y))^2 \geq \sum_{\{x,y\} \in E(T_{h,k})} (\bar f(x)-\bar f(y))^2\,.
\end{equation}
This follows by applying Jensen's inequality to the edges from $V_{\ell}$ to $V_{\ell+1}$ for each value of $\ell$.
Now summing lines \eqref{eq1} and \eqref{eq3} and using \eqref{eq2} yields
$$
\sum_{\{x,y\} \in E(\hat T_{h,k})} (f(x)-f(y))^2 \geq \frac{1}{k^2} \left(\frac{\|\bar f\|^2}{6} + \|f-\bar f\|^2\right)
\geq \frac{1}{7 k^2} \|f\|^2\,,
$$
completing the proof.
\end{proof}

\section{Conclusion}

As Yuval Peres pointed out to us, one can check that the mixing time of simple random walk on
$\hat T_{h,k}$ is on the order of $hk^2$ for $k=2^h$, which is a $\log |V(\hat T_{h,k})|$ factor larger
than the relaxation time (i.e., the inverse spectral gap).  Observe that although the diameter of $\hat T_{h,k}$
is $hk$, the average distance between a uniformly random pair of vertices is bounded by $O(k)$.
And indeed, it is true in general that in a bounded-degree
planar graph, the mixing time is at least the average of the squared distance.

\begin{theorem}[\cite{BLR10}]
For some constant $c > 0$, the following holds.
Let $G=(V,E)$ be a planar graph with path metric $d$.
Then,
$$
\lambda_1(G) \leq c \left(\frac{1}{|V|^2} \sum_{x,y \in V} d(x,y)^2\right)^{-1}\,.
$$
\end{theorem}

\begin{proof}
By \cite[Thm 4.4]{BLR10}, there exists a universal constant $C > 0$ and a 1-Lipschitz mapping $f : V \to \mathbb R$
such that $$\sum_{x,y \in V} |f(x)-f(y)|^2 \geq C \sum_{x,y \in V} d(x,y)^2\,.$$
By a standard calculation using the fact that $|E| \leq 3|V|$, one has
$$
\lambda_1(G) \leq \frac{\sum_{\{x,y\} \in E} |f(x)-f(y)|^2}{\frac{1}{2|V|} \sum_{x,y \in V} |f(x)-f(y)|^2} \leq \frac{(3|V|)(2|V|)}{C \sum_{x,y \in V} d(x,y)^2}\,,
$$
completing the proof.
\end{proof}
We remark that a similar bound holds for any graph of bounded genus, or more generally, for any graph excluding a fixed minor; see \cite{BLR10}.
For families of graphs of unbounded degree, one should consider the normalized Laplacian $\mathcal L f(x) = \sum_{y : \{x,y\} \in E} (\pi(x) \pi(y))^{-1/2} (f(x)-f(y))$,
where $\pi(x)$ denotes the stationary probability of $x \in V$.  A similar argument shows that for some $c > 0$,
$$\lambda_1(\mathcal L) \leq c\left(\sum_{x,y \in V} \pi(x) \pi(y) d(x,y)^2\right)^{-1}\,,$$
implying that, in planar graphs, the mixing time is at least the average squared distance
when points are chosen uniformly from the stationary measure.

\bibliographystyle{abbrv}
\bibliography{gap}

\end{document}